\documentclass[12pt,a4paper]{article}
\usepackage{graphicx}
\usepackage{amsmath,amsthm,amsfonts} 
\usepackage{txfonts,bm,eucal} 

    \DeclareMathOperator\id{id}
    
    \DeclareMathOperator\GL{GL}

    \DeclareMathOperator\dis{\triangle}
    \DeclareMathOperator\notdis{{\not\!\!\dis}}

    \DeclareMathOperator\End{End}
\DeclareMathOperator\Aut{Aut}

\newcommand{\cG}{{\mathcal G}}
\newcommand{\cH}{{\mathcal H}}
\newcommand{\cM}{{\mathcal M}}

\newcommand{\cR}{{\mathcal R}}
\newcommand{\cA}{{\mathcal A}}

\newcommand{\cC}{{\mathcal C}}
\newcommand{\cS}{{\mathcal S}}
\newcommand{\cT}{{\mathcal T}}

\newcommand{\bP}{{\mathbb P}}

\newcommand{\fC}{{\mathfrak C}}
\newcommand{\fP}{{\mathfrak P}}
\newcommand{\fR}{{\mathfrak R}}

\newcommand\csi{\cS_\mathrm{I}}
\newcommand\csii{\cS_\mathrm{II}}

\let\phi=\varphi

\newcommand{\spmatrix}[1]
{\mbox{\scriptsize\setlength\arraycolsep{0.5\arraycolsep}$\begin{pmatrix}#1\end{pmatrix}$}}

\newtheorem{lem}{Lemma}[section]
\newtheorem{thm}[lem]{Theorem}
\newtheorem{prop}[lem]{Proposition}
\newtheorem{cor}[lem]{Corollary}
\newtheorem*{pro}{Problem}
{\theoremstyle{definition}
\newtheorem{exa}[lem]{Example}
\newtheorem{defi}[lem]{Definition}
}
{\theoremstyle{remark}
\newtheorem{rem}[lem]{Remark}
}

\newenvironment{cond}[1]{%
\renewcommand{\labelenumi}{{\rm ({#1}\arabic{enumi})}}%
\renewcommand{\theenumi}{{({#1}\arabic{enumi})}}%
\settowidth\labelwidth{({#1}{1})}\setlength\leftmargini{\labelwidth}\addtolength\leftmargini{\labelsep}%
\begin{enumerate}}{\end{enumerate}}

\begin{document}

\title{Geometric structures on finite- and infinite-dimensional Grassmannians}

\author{Andrea Blunck \and Hans Havlicek}

\maketitle

\newcommand\xxx[1]{\marginpar{\footnotesize HH: #1}}

\begin{abstract}
In this paper, we study the Grassmannian of $n$-dimensional subspaces of a
$2n$-dimensional vector space and its infinite-dimensional analogues. Such a
Grassmannian can be endowed with two binary relations (adjacent and distant),
with pencils (lines of the Grassmann space) and with so-called $Z$-reguli. We
analyse the interdependencies among these different structures.
\end{abstract}

\section{Introduction}\label{se:intro}

Let $V$ be a left vector space of arbitrary (not necessarily finite) dimension
over an arbitrary (not necessarily commutative) field $K$. It will always be
assumed that $\dim V>2$. We study the set\footnote{We use the sign $\leq$ for
the inclusion of subspaces and reserve $<$ for strict inclusion.}
\begin{equation*}
    \cG:=\{X\le V\mid X\cong V/X\}
\end{equation*}
of subspaces $X$ of $V$ that are isomorphic to the quotient space $V/X$.
Clearly, this condition is equivalent to saying that $X$ is isomorphic to one
(and hence all) of its complements. We assume that $\cG\neq\emptyset$. So, if
$\dim V$ is finite, then it is an even number $2n$, say, and $\cG$ is just the
Grassmannian of $n$-dimensional subspaces of $V$.
\par

The set $\cG$ can be endowed with several structures such that $\cG$ becomes
the vertex set of a graph or the point set of an incidence geometry. We
investigate the interrelations among these structures and among their
automorphism groups. Section~\ref{se:grass} is devoted to the \emph{adjacency
relation} on $\cG$ and its associated \emph{Grassmann graph}. The
\emph{Grassmann space} on $\cG$ is based on the notion of a \emph{pencil} of
subspaces. We extend results about these two structures, which are well known
for $\dim V<\infty$, to infinite dimension. In Section~\ref{se:distant} we
recall the \emph{distant relation} on $\cG$, where ``distant'' is just another
phrase for ``complementary'', the \emph{distant graph}, and the link with
\emph{chain geometries}. The \emph{$Z$-reguli} (reguli over the centre $Z$ of
the ground field $K$) from Section~\ref{se:reguli} are distinguished subsets of
$\cG$. Our main result (Theorem~\ref{thm:defbar}) says that $Z$-reguli can be
defined in terms of the distant graph. The key tool is a characterisation of
$Z$-reguli in Theorem \ref{thm:Zreg} and a description of adjacency in terms of
the distant graph from \cite{blu+h-03b}. Finally, in Section~\ref{se:conseq} we
state a series of corollaries about automorphisms.

\par

Throughout the paper we prefer the projective point of view, using the language
of points and lines for one- and two-dimensional subspaces. Lower case letters
are reserved for points, the join of subspaces is denoted by $+$. Note that
dimensions are always understood in terms of vector spaces (rather than
projective dimensions).
\par
Although there is no principle of duality for infinite-dimensional vector
spaces, for \emph{some\/} statements in this article a \emph{dual statement\/}
can be obtained as follows: (i)~Reverse all inclusion signs between subspaces.
(ii)~Change the order of subspaces defining a quotient space (e.g.\ $X/Y$ turns
into $Y/X$). (iii)~For any integer $k\geq 0$ replace ``subspace of $V$ with
dimension $k$'' by ``subspace of $V$ with codimension $k$'' and vice versa.
(iv)~Interchange signs for join and meet of subspaces.
\par
For example, conditions (\ref{eq:adjacent1}) and (\ref{eq:adjacent2}) (see
below) are dual to each other. We shall frequently claim that the dual of a
certain result holds. In such case the reader will easily verify that the proof
of the dual result can be accomplished by dualising the initial proof. Clearly,
in case of finite dimension this is a consequence of the usual principle of
duality, otherwise this is due to the specific content of the initial result.

\section{Grassmann graph and Grassmann space}\label{se:grass}

Two elements $X,Y\in\mathcal G$ are called \emph{adjacent} (in
symbols: $X\sim Y$) if
\begin{equation}\label{eq:adjacent1}
    \dim((X+Y)/X) = \dim((X+Y)/Y)=1,
\end{equation}
or, equivalently, if
\begin{equation}\label{eq:adjacent2}
    \dim (X/(X\cap Y)) = \dim (Y/(X\cap Y)) =1.
\end{equation}
This terminology goes back to \textsc{W.-L.\ Chow} \cite{chow-49} in the
finite-dimensional case. Clearly, adjacency is an antireflexive and symmetric
relation. The \emph{Grassmann graph} on $\cG$ is the graph whose vertex set is
$\cG$ and whose edges are the $2$-sets of adjacent vertices. It is studied
(also in the infinite-dimensional case) e.g.\ in \cite{blu+h-03b},
\cite[3.8]{pankov-10a}. In the finite-dimensional case, $\cG$ can also be
viewed as the point set of a \emph{projective geometry of matrices} (compare
\cite[3.6]{wan-96}).

\par
Let $M\le V$ be a subspace such that there is an $X\in \cG$ with $M\le X$ and
$\dim (X/M)=1$. We define the set
\begin{equation*}
    \cG[M\rangle:=\{E\le V\mid M\leq E \mbox{~~and~~} \dim    (E/ M) =1 \}
\end{equation*}
and call it the \emph{star} with centre $M$. Dually, given an $N\le V$ for
which there exists an $X\in\cG$ with $X\le N$ and $\dim (N/X)=1$, we set
\begin{equation*}
    \cG\langle N]:=\{E\le V\mid E\le N \mbox{~~and~~}  \dim (N/E) =1 \}
\end{equation*}
and call it the \emph{top} with carrier $N$.\

\par

Recall our global assumption $\dim V> 2$. It guarantees that a set of subspaces
of $V$ cannot be at the same time a top and a star: The subspaces of any star
cover the entire space $V$, whereas the elements of any top, say $\cG\langle
N]$, cover only the proper subspace $N < V$. Note also that the star with
centre $M$ coincides with the set $\cM:=\{E\in\cG\mid M\le E\}$ only for $\dim
V<\infty$, because here $E\in\cM$ implies $\dim (E/M)=1$. On the other hand, in
the infinite-dimensional case for any integer $n\ge 0$ there exists at least
one $E_n\in \cM$ with $\dim (E_n/M)=n$. It can be obtained as the join of $M$
with $n$ independent points in a complement of $M$.
\par
A set of mutually adjacent elements from $\cG$ is nothing but a clique of the
Grassmann graph. It will also be called an \emph{adjacency clique}. Our first
aim is to show that stars and tops are the maximal adjacency cliques, a fact
which is well known in the finite-dimensional case (see, e.g.,
\cite[Prop.~3.2]{pankov-10a}).

\begin{lem}\label{lem:Madj}
Let $\cG[M\rangle$ be a star. Then the following hold:
\begin{enumerate}
\item $\cG[M\rangle\subseteq\cG$.
\item Any two distinct elements $E,E'\in\cG[M\rangle$ are adjacent.
\item Two adjacent elements $E,E'\in\cG$ belong to $\cG[M\rangle$ if, and
    only if, $E\cap E'=M$.
\end{enumerate}
\end{lem}
\begin{proof}
Ad 1.: By definition, $\dim (M/X)=1$ for some $X\in\cG$. Given any
$E\in\cG[M\rangle$ we have $\dim E = \dim M + 1 = \dim X$ and $\dim (V/E) =
\dim M -1 =\dim (V/X)$. So $E\cong X\cong V/X \cong V/E$.

\par
Ad 2.: Let $E,E'$ be distinct elements of $\cG[M\rangle$. Since $E=E\cap E'=E'$
is impossible, we may assume w.l.o.g.\ that $E\cap E'< E$. From this and the
definition of $\cG[M\rangle$, we obtain $M\le E\cap E'< E$. Now $\dim (E/M)=1$
yields $M=E\cap E'$, which implies $E\sim E'$ due to $\dim (E'/M) = 1$ and
(\ref{eq:adjacent2}).

\par
Ad 3.: The ``if part'' is trivial, the ``only if part'' follows from the proof
of 2.
\end{proof}
Dually, the following can be proved:
\begin{lem}\label{lem:Nadj}
Let $\cG\langle N]$ be a top. Then the following hold:
\begin{enumerate}
\item $\cG\langle N]\subseteq\cG$.
\item Any two distinct $E,E'\in\cG\langle N]$ are adjacent.
\item Two adjacent elements $E,E'\in\cG$ belong to $\cG\langle N]$ if, and
    only if, $E+ E'=N$.
\end{enumerate}
\end{lem}

\begin{lem}\label{lem:3adj}
Let $A,B,C\in\cG$ be mutually adjacent. Then there is a star or a top
containing them.
\end{lem}
\begin{proof}
Assume that $A,B,C$ do not belong to any star. This means by Lemma
\ref{lem:Madj}.3 that w.l.o.g.\ $A\cap B\ne A\cap C$. Let $a\leq A$ and $b\leq
A\cap B$ be points with $a\not\subseteq B\cup C$ and $b\not\leq C$. Then the
line $L=a+b\leq A$ does not lie in $A\cap C$, which is a hyperplane of $A$ due
to $\dim (A/(A\cap C))=1$. Consequently, $L$ meets $A\cap C$ in a point $c\ne
b$. So $a\leq L=b+c\leq B+C$. Altogether, $A\le B+C$. This implies $B<A+B\le
B+C$, whence $A+B=B+C$ as $B$ is a hyperplane in $B+C$. Analogously, $A+C=B+C$.
So $A,B,C$ belong to the top $\cG\langle N]$ with $N:=A+B$.
\end{proof}

\begin{prop}\label{prop:adj-cliques}
The maximal adjacency cliques of the Grassmannian $\cG$ are precisely the stars
and tops.
\end{prop}
\begin{proof}
(a) \emph{We show that any adjacency clique $\cA\subseteq\cG$ is a subset of a
star or a subset of a top.} For $|\cA|<2$ the assertion obviously holds.
Otherwise there exist two distinct elements $A, B \in \cA$. We read off from
Lemma~\ref{lem:Madj}.3 and Lemma~\ref{lem:Nadj}.3 that they belong to the star
$\cG[A \cap B\rangle =: \cS$ and to the top $\cG\langle A + B] =: \cT$. If
$\cA$ is contained in $\cS\cap \cT$ then we are done. Otherwise there exists a
$C\in \cA\setminus(\cS\cap \cT)$. We infer from this and from
Lemma~\ref{lem:3adj}, applied to $A, B, C$, that $C$ belongs to the symmetric
difference of $\cS$ and $\cT$. Hence there are two cases:

\par
\emph{Case 1.} $C \in \cS \setminus \cT$: We claim that $\cA \subseteq \cS$.
For if there were an $X \in \cA \setminus \cS$ then we could apply
Lemma~\ref{lem:3adj} first to $A, B, X$ and then to $A,C, X$. This would give
$X \in \cG\langle A + B] \cap \cG\langle A + C] = \{A\} \subset \cS$, an
absurdity.
\par
\emph{Case 2.} $C \in \cT \setminus \cS$: Here $\cA \subseteq \cT$ follows
dually to Case 1.

\par
(b) Let $\cS=\cG[M\rangle$ be any star. By Lemma~\ref{lem:Madj}.2, the star
$\cS$ is an adjacency clique. Furthermore, let $\cA\subseteq\cG$ be an
adjacency clique containing $\cS$. We infer from (a) that $\cA$ is a subset of
a star or a subset of a top. However the latter cannot occur, because there is
no top containing $\cS$. So there is a star, say $\cS' = \cG[M'\rangle$, with
$\cS\subseteq \cA \subseteq\cS'$. There are two distinct elements $A,B\in\cS$.
From Lemma~~\ref{lem:Madj}.3 we infer $M=A\cap B = M'$. Hence $\cS = \cA =
\cS'$ which shows that $\cS$ is a maximal adjacency clique.
\par
Dually, any top is a maximal adjacency clique.
\par
(c) Given any maximal adjacency clique $\cA\subseteq\cG$, it is contained in a
star $\cS$ or in a top $\cT$ by (a). The maximality of $\cA$ implies that
$\cA=\cS$ or $\cA=\cT$.
\end{proof}
Let $M,N$ be subspaces of $V$ such that
\begin{equation}\label{eq:M-N}
    \mbox{ there is an } X\in\cG \mbox{ with }
  M\leq X\leq N \mbox{ and } \dim(X/M)=\dim(N/X)=1.
\end{equation}
Then
\begin{equation}\label{eq:bueschel}
  \cG[M,N]:=\{X\in\cG\mid M< X< N\}
\end{equation}
is called a \emph{pencil} in $\cG$. If $\dim V=2n$ is finite then
(\ref{eq:M-N}) is equivalent to $M\leq N$, $\dim M=n-1$, and $\dim N=n+1$.

\par
Any pencil $\cG[M,N]$ is contained in the star $\cG[M\rangle$. As stars are
adjacency cliques of $\cG$ by Lemma~\ref{lem:Madj}.1, so are pencils. Let $\fP$
denote the set of all pencils. Then $(\cG,\fP)$ can be viewed as a partial
linear space, i.e.\ a point-line incidence geometry with ``point set'' $\cG$
and ``line set'' $\fP$ such that any two ``points'' are joined by at most one
``line''. This geometry in called the \emph{Grassmann space} on $\cG$ (see, for
example, \cite[3.1]{pankov-10a} for the finite-dimensional case).

\par
Two distinct elements $X,Y$ of $\cG$ are adjacent if, and only if, they are
``collinear'' in $(\cG,\fP)$, i.e., if they belong to a common pencil (which
then has to be $\cG[X\cap Y,X+Y]$). Stars and tops are the maximal singular
subspaces of $(\cG,\fP)$, i.e., subspaces in which any two distinct points are
collinear. More precisely, they are projective spaces. An underlying vector
space of a star $\cG[M\rangle$ is the quotient space $V/M$, whereas for a top
$\cG \langle N]$ the dual space $N^*$ of $N$ plays this role. Consequently, all
``lines'' of the Grassmann space $(\cG,\fP)$ contain at least three (actually
$|K|+1$) ``points''.

\par
We saw in the preceding paragraph that the adjacency relation $\sim$ can be
defined using the concept of pencil only. The subsequent Theorem
\ref{thm:penadj} implies that pencils can be defined in $\cG$ by using the
relation $\sim$ only.

\begin{thm}\label{thm:penadj}
The pencils of the Grassmann space $(\cG,\fP)$ are exactly the sets with more
than one element that are intersections of two distinct maximal adjacency
cliques.
\end{thm}

\begin{proof}
Given a pencil as in (\ref{eq:bueschel}) we noted already that $|\cG[M,N]|\ge
3$. The second required property follows from $\cG[M,N]=\cG[M\rangle \cap
\cG\langle N]$ and Proposition~\ref{prop:adj-cliques}.

\par
Conversely, let $\cS$ with $|\cS|\ge 2$ be the intersection of two maximal
adjacency cliques. By Proposition~\ref{prop:adj-cliques} we are led to the
following cases:
\par
\emph{Case 1.} $\cS$ is the intersection of two distinct stars, say
$\cS=\cG[M\rangle\cap\cG[M'\rangle$. Choose an $E\in \cS$. Then $M<M+M'\le E$,
so $E=M+M'$ as $\dim (E/M)=1$. This means that $E$ is uniquely determined, a
contradiction. Dually, $\cS$ cannot be the intersection of two distinct tops.
\par
\emph{Case 2.} $\cS$ is the intersection of a star and a top, say
$\cS=\cG[M\rangle\cap\cG\langle N]$. If $M\not\le N$ then
$\cG[M\rangle\cap\cG\langle N]=\emptyset$ which is impossible. So $M\le N$ and
hence $\cG[M\rangle\cap\cG\langle N]=\cG[M,N]$.
\end{proof}
Up to here the dimension of $V$ did not play an essential role. Yet there are
properties of the Grassmann graph and the Grassmann space on $\cG$ which depend
on the dimension of $V$ being finite or not.

\begin{rem}\label{rem:disconn}
By \cite[2.3]{blu+h-03b}, the Grassmann graph $(\cG,\sim)$ is connected if, and
only if, $\dim V=2n<\infty$. In this case the diameter of the graph is $n$. For
infinite dimension of $V$ the connected component of $X\in\cG$ equals
    \begin{equation}\label{eq:zsh}
        \{E\in\cG \mid \dim (E/(E\cap X)) = \dim (X/(E\cap X)) <\infty\}
    \end{equation}
and its diameter is infinite.
\end{rem}

\section{Distant graph and chain geometries}\label{se:distant}

We say that $X,Y\in\cG$ are \emph{distant} (in symbols: $X\,\dis\,Y$) whenever
they are complementary, i.e., $X\oplus Y=V$. Also this is an antireflexive and
symmetric relation. The \emph{distant graph} on $\cG$ is the graph whose vertex
set is $\cG$ and whose edges are the $2$-sets of distant vertices. See
\cite{blu+h-01a}, \cite{blu+h-03b}. The cliques of the distant graph will be
called \emph{distant cliques}.

\begin{rem}\label{rem:disadj}
In \cite{blu+h-03b} the following is proved:
\begin{enumerate}
\item\label{rem:disadj.1} The relation $\sim$ can be defined by using
    $\dis$ only \cite[Thm.~3.2]{blu+h-03b}: Two different elements
    $A,B\in\cG$ are adjacent if, and only if, there is a $C\in
    \cG\setminus\{A,B\}$ such that for all $X\in\cG$ with $X\dis C$ also
    $X\dis A$ or $X\dis B$ holds.
\item If $\dim V=2n<\infty$, the relation $\dis$ can be defined by $\sim$
    only: The elements $X,Y\in\cG$ are distant if, and only if, the
    distance of $X$ and $Y$ in the Grassmann graph on $\cG$ equals $n$
    (which is the diameter of the Grassmann graph); this follows from
    formula (3) in \cite{blu+h-03b}. Therefore some authors
     speak of \emph{opposite} rather than distant vertices of the Grassmann
    graph. Cf.\ \cite[3.2.4]{pankov-10a}.
\item If $\dim V=\infty$, the relation $\dis$ cannot be defined by $\sim$
    only: There are permutations of $\cG$ leaving $\sim$ invariant but
    \emph{not} leaving $\dis$ invariant \cite[Ex.~4.3]{blu+h-03b}.
     (E.g., the $\kappa$ from Example \ref{exa:nodistiso} has this
     property.)
\end{enumerate}
\end{rem}
The relation ``distant'' comes from ring geometry; see, among others,
\cite[p.~15]{blu+he-05}, \cite[Def.~1.2.1]{herz-95}, and
\cite[Def.~3.1]{veld-95}. We therefore recall some definitions and results. For
any associative ring $S$ with $1$ the general linear group $\GL(2,S)$ acts on
the free left module $S^2$ and on the lattice of its submodules. The
\emph{projective line over $S$} is the orbit
\begin{equation*}
  \bP(S):=S(1,0)^{\GL(2,S)}
\end{equation*}
of the free cyclic submodule $S(1,0)$ under this action. On $\bP(S)$, the
antireflexive and symmetric relation $\dis$ (\emph{distant}) is defined by
\begin{equation*}
  \dis:=(S(1,0),S(0,1))^{\GL(2,S)}
\end{equation*}
See \cite{blu+he-05} or \cite{herz-95} for a detailed exposition. We now adopt
the additional assumption that $S$ contains a field $F$ (with $1_F=1_S$) as a
proper subring. Then $\bP(F)$ can be embedded in $\bP(S)$ via $F(a,b)\mapsto
S(a,b)$. The orbit
\begin{equation*}
     \fC(F,S):=\bP(F)^{\GL(2,S)}
\end{equation*}
is called the set of \emph{$F$-chains} in $\bP(S)$, and the incidence geometry
$\Sigma(F,S):=(\bP(S),\fC(F,S))$ is called the \emph{chain geometry} over
$(F,S)$.

\par
Originally, chain geometries have been studied in the case that $S$ is an
$F$-algebra, i.e., the field $F$ is contained in the centre of $S$ (see
\cite{blu+he-05}, \cite{herz-95}). Then, given three mutually distant points,
there is a \emph{unique} chain containing them. If $F$ is not in the centre of
$S$, then, in general, there is more than one chain through three mutually
distant points. See \cite{blu+h-00a}, where we used the term ``generalized
chain geometry'' in order to emphasise the deviations from the original
setting. The crucial observation for us is as follows:

\begin{rem}\label{rem:dist-chain}
Two distinct points of $\bP(S)$ are distant if, and only if, they are on a
common $F$-chain \cite[Lemma~2.1]{blu+h-00a}.
\par

Observe that this characterisation provides a definition of the distant
relation in terms of $F$-chains. It does not depend on the chosen field
$F\subset S$.
\end{rem}
The set $\cG$ can be interpreted as the projective line over the endomorphism
ring of a vector space:

\begin{rem}\label{rem:isoP(R)}
Let $U$ and $U'$ be arbitrary distant elements of $\cG$, and let $R=\End_K(U)$
be the endomorphism ring of $U$. Furthermore let $\lambda:U\to U'$ be a linear
isomorphism. By \cite[Thm.~2.4]{blunck-99} the following assertions
hold:\footnote{The results in \cite{blunck-99} are stated in terms of $U\times
U$. We rephrase them by virtue of the linear isomorphism which maps
$(u_0,u_1)\in U\times U$ to $u_0+u_1^\lambda\in V=U\oplus U'$.}
\begin{enumerate}
\item\label{rem:isoP(R).1} The mapping
    \begin{equation}\label{eq:Phi}
      \Phi:\bP(R) \to \cG: R(\alpha,\beta)\mapsto
      U^{(\alpha,\,\beta)}=\{u^\alpha+ u^{\beta\lambda}\mid u\in U\}
    \end{equation}
    is a well defined bijection.
\item\label{rem:isoP(R).2} Points $p,q\in\bP(R)$ are distant if, and only
    if, their images $p^\Phi,q^\Phi$ are distant (i.e.\ complementary) in
    $\cG$.
\item\label{rem:isoP(R).3} $\Phi$ induces an isomorphism of group actions
    \begin{equation*}
        (\bP(R),\GL(2,R))\to (\cG,\Aut_K(V))
    \end{equation*}
    as follows: For any
    $\psi=\spmatrix{\alpha&\beta\\\gamma&\delta}\in\GL(2,R)$ let
    $\hat\psi:V\to V$ be defined by
    \begin{equation*}
        u_0+u_1^\lambda\mapsto (u_0^\alpha+u_1^{\gamma}) + (u_0^{\beta} +
    u_1^{\delta})^\lambda\mbox{~~for all~~} u_0,u_1\in U.
    \end{equation*}

     Then $\GL(2,R)\to\Aut_K(V): \psi\mapsto\hat\psi$ is an isomorphism of
    groups satisfying $\psi\Phi=\Phi\hat\psi$.
\end{enumerate}
\end{rem}
In the case that $\dim V=2n<\infty$ we can identify $R$ with the ring of
$n\times n$ matrices over $K$. Then (\ref{eq:Phi}) shows once more that the
Grassmannian $\cG$ can be identified with the point set of the projective
geometry of square matrices studied in \cite{wan-96}, since
$U^{(\alpha,\beta)}$ equals the left row space of the $n\times 2n$ matrix
$(\alpha,\beta)$.

\par
The definition of $\Phi$ in (\ref{eq:Phi}) relies on the choice of $U$, $U'$,
and $\lambda$. However, this choice is immaterial: For, if we select instead
any two distant elements $\tilde U$, $\tilde U'$ of $\cG$ and a linear
isomorphism $\tilde \lambda: \tilde U\to \tilde U'$ then we obtain a bijection
$\tilde\Phi$ of the projective line over the endomorphism ring $\tilde R$ of
$\tilde U$ onto $\cG$ like the one in (\ref{eq:Phi}). There exists a linear
isomorphism $\iota:U\to\tilde U$, whence the mapping $R\to\tilde
R:\alpha\mapsto\iota^{-1}\alpha\iota=:\tilde\alpha$ is an isomorphism of rings,
and the bijection $\bP(R)\to\bP(\tilde R):R(\alpha,\beta)\mapsto \tilde
R(\tilde\alpha,\tilde\beta)$ takes distant points to distant points in both
directions. Further, the linear automorphism $V\to V: u_0+u_1^\lambda\mapsto
u_0^\iota+u_1^{\iota\tilde\lambda}$ (with $u_0,u_1\in U$) will send any
$R(\alpha,\beta)^\Phi\in\cG$ to $\tilde
R(\tilde\alpha,\tilde\beta)^{\tilde\Phi}\in\cG$. (This observation generalises
Remark~\ref{rem:isoP(R)}.\ref{rem:isoP(R).3}.)

\begin{rem}\label{rem:durchmesser}
By virtue of (\ref{eq:Phi}) we obtain the following: The distant graph on $\cG$
is connected; it has diameter $3$ for $\dim V=\infty$
\cite[Thm.~5.3]{blu+h-01a} and diameter $2$ for $\dim V<\infty$. The second
part of the last assertion is immediate from \cite[Prop.~1.1.3]{herz-95} and
\cite[2.6]{veld-85}.
\end{rem}
We shall see below that the projective line $\bP(R)$, $R=\End_K(U)$, can be
considered as the point set of a chain geometry $\Sigma(F,R)$ in at least one
way. As we noted above, the distant relation is then definable in terms of
$F$-chains. Taking into account Remark~\ref{rem:isoP(R)}, the following
converse question arises:

\begin{pro}
Given a subfield $F$ of the endomorphism ring $R=\End_K(U)$ is it possible to
define the $\Phi$-images of $F$-chains in terms of the distant graph
$(\cG,\dis)$?
\end{pro}
A major obstacle in solving this problem is that for arbitrary $F$ an explicit
description of the $\Phi$-images of $F$-chains even in terms of the projective
space on $V$ seems to be unknown. The situation seems less intricate for the
following class of examples: Let $F$ be any subfield of $K$. We embed $F$ in
$R$ by fixing a basis $(b_i)_{i\in I}$ of $U$ and mapping $a\in F$ to the
unique endomorphism of $U$ with $b_i\mapsto a b_i$ for all $i\in I$. The case
$F=K$ was detailed in \cite{blunck-00a}. Here the $\Phi$-images of $K$-chains
are \emph{reguli}. However the definition of a regulus in \cite{blunck-00a} is
rather involved in its most general form\footnote{In
\cite[Def.~2.3]{blunck-00a} the following minor revision has to be made in
order to assure the results from \cite{blunck-00a}: Replace the assumption that
$(T_i)_{i\in I}$ is a minimal set of lines generating the vector space $V$ by
the stronger assumption that $(T_i)_{i\in I}$ is a family of lines such that
$V=\bigoplus_{i\in I}T_i$.}, i.e., when $K$ is a proper skew field and $\dim
U=\infty$. We therefore focus on the case when $F$ equals the centre $Z$ of
$K$. Here the choice of a basis from before is immaterial, since the
endomorphism corresponding to $a\in Z$ is simply
\begin{equation}\label{eq:a.id}
    a\cdot\id\in R=\End_K(U).
\end{equation}
The subsets of $\cG$ that correspond under $\Phi$ to $Z$-chains will be
exhibited in the next section.

\section{$Z$-Reguli}\label{se:reguli}

We start with a definition of $Z$-reguli. Their connection with the
$\Phi$-images of $Z$-chains will only be shown in Theorem~\ref{thm:Zchain}.
Note that most of the following proofs are considerably easier in the case of
finite dimension.

\begin{defi}\label{def:regulus}
A \emph{$Z$-regulus} is a subset $\cR$ of $\cG$ satisfying the following
conditions:
\begin{cond}{R}
\item\label{def:regulus.1} $\cR$ is a distant clique with at least three
    elements.

\item\label{def:regulus.2} If three mutually distinct elements of $\cR$
    meet\footnote{We say that two subspaces of $V$ \emph{meet} each other if
    they have a common point.} a line then all elements of $\cR$ meet that
    line.

\item\label{def:regulus.3} $\cR$ is not properly contained in any subset of
    $\cG$ satisfying conditions \ref{def:regulus.1} and \ref{def:regulus.2}.
\end{cond}
\end{defi}
During our investigation we shall frequently come across subsets of $\cG$ that
satisfy conditions \ref{def:regulus.1} and \ref{def:regulus.2}, but not
necessarily the maximality condition \ref{def:regulus.3}. Such a set will be
termed as being a \emph{partial $Z$-regulus}. A line $L$ that meets all
elements of a partial $Z$-regulus $\cR$ is called a \emph{directrix} of $\cR$.
Note that this does not necessarily mean that each point of $L$ is on some
element of $\cR$.

\begin{rem}\label{rem:segre}
Let $U$ and $U'$ be distant elements of $\cG$ and let $\lambda:U\to U'$ be a
linear isomorphism. There are two distinguished families of subspaces of $V$
which are entirely contained in the set
\begin{equation*}
    Q:=\{r u+ s u^\lambda\mid u\in U,\; r,s\in K\}.
\end{equation*}
The first family $\csi$ comprises all subspaces of the form
\begin{equation}\label{eq:1st}
    L_u:=\{ru+su^\lambda\mid  r,s\in K\} \mbox{~~with~~} u\in U\setminus\{0\},
\end{equation}
and we call them subspaces of \emph{first kind}. The second family $\csii$ is
formed by the subspaces of \emph{second kind}. They are given as
\begin{equation}\label{eq:T(x,y)}
    T^{(x,\,y)}:=\{xu+yu^\lambda\mid u\in U\} \mbox{~~with~~}  (x,y)\in Z^2\setminus\{(0,0)\}.
\end{equation}
Up to minor notational differences the following was shown in
\cite[Thm.~1]{havl-93a}:
\begin{enumerate}
\item\label{rem:segre.1} The subspaces of second kind are precisely the
    \emph{transversal subspaces} of $\csi$, i.e., those subspaces $T\leq V$
    for which a bijection of $\csi$ to the point set of $T$ is given by the
    assignment $L\, (\in\csi)\mapsto L\cap T$.

\item\label{rem:segre.2} If $L\subseteq Q$ is a line then either $L$ is a
    subspace of first kind or $L$ is contained in a subspace of second
    kind.
\end{enumerate}
The first result can be rephrased as follows: Any two subspaces of different
kind have a unique point in common. Each point which is on some subspace of
second kind is on a unique subspace of first kind. (Compare also with
\cite[Satz~10.1.4]{blu+he-05}, where similar results are derived under stronger
assumptions.) We add in passing that in \cite{havl-93a} the set of all points
$p$ with $p\subseteq Q$ is called a \emph{Segre manifold}. However, we shall
not be concerned with this notion.
\end{rem}
In order to prove that $\csii$ is a $Z$-regulus we need an auxiliary
result.

\begin{lem}\label{lem:Zreg}
Let $E$ be any element of a partial $Z$-regulus $\cR$. Then a bijection from
the set of directrices of $\cR$ to the point set of $E$ is given by the
assignment $L\mapsto L\cap E$. Consequently, each point of $E$ is on a unique
directrix of $\cR$.
\end{lem}
\begin{proof}
There are $E', E''\in \cR \setminus \{E\}$ with $E'\ne E''$ and hence $E'\dis
E''$.

\par
First, let $L$ be a directrix of $\cR$. As $L$ meets the distant subspaces $E$
and $E'$ the intersection $L\cap E$ is a point. Thus $L\mapsto L\cap E$ gives a
well-defined mapping from the set of directrices of $\cR$ to the point set of
$E$.

\par
Next, let $p\leq E$ be a point. We have $V=E'\oplus E''$, $p\not\leq E'$, and
$p\not\leq E''$. So there is a unique line $L'$ through $p$ meeting $E'$ and
$E''$. Since $\cR$ is a partial $Z$-regulus, this line $L'$ is a directrix of
$\cR$ and, by the uniqueness of $L'$, no other directrix of $\cR$ can pass
through $p$. Hence our mapping is bijective.
\end{proof}
From now on we assume the bijection $\Phi:\bP(R)\to \cG$ to be given in terms
of $U$, $U'$, and $\lambda$ as in Remark~\ref{rem:isoP(R)}. We use the same
$U$, $U'$, and $\lambda$ to define the notions from Remark~\ref{rem:segre}.

\begin{prop}\label{prop:Sii}
The set $\csii$ comprising all subspaces $T^{(x,\,y)}$ from
\emph{(\ref{eq:T(x,y)})} is, on the one hand, the $\Phi$-image of a $Z$-chain
and, on the other hand, a $Z$-regulus.
\end{prop}

\begin{proof}
(a) We consider the $Z$-chain $\cC$ which arises by embedding $\bP(Z)$ in
$\bP(R)$ via $Z(x,y)\mapsto R(x,y)$ (cf.\ Remark~\ref{rem:disadj}) and obtain
\begin{equation*}
    R(x,y)^\Phi = U^{(x\,\cdot\,\id,\,y\,\cdot\,\id)}\mbox{~~~~for all~~~~}
    (x,y)\in Z^2\setminus\{(0,0)\}.
\end{equation*}
Comparing (\ref{eq:Phi}) and (\ref{eq:a.id}) with (\ref{eq:T(x,y)}) yields
\begin{equation*}
    U^{(x\,\cdot\,\id,\, y\,\cdot\,\id)}=T^{(x,\,y)}\mbox{~~~~for all~~~~} (x,y)\in Z^{2}\setminus\{(0,0)\},
\end{equation*}
whence $\csii=\cC^\Phi$.

\par
(b) Taking into account Remark~\ref{rem:isoP(R)}.\ref{rem:isoP(R).2} and the
fact that the points of the $Z$-chain $\cC$ are mutually distant (see
Remark~\ref{rem:dist-chain}), the elements of $\csii$ turn out to be mutually
distant. Together with $3\leq|\cC|=|\csii|$ this shows that $\csii$ satisfies
condition \ref{def:regulus.1}.

\par
Suppose now that a line $L$ meets three distinct elements $E_0$, $E_1$, and $E$
of $\csii$. Then $L\cap E$ is a point and, by
Remark~\ref{rem:segre}.\ref{rem:segre.1}, there is a unique line $L_u\in\csi$
through $L\cap E$. The line $L_u$ meets all elements of $\csii$. Since there is
a unique line through $p$ which meets $E_0$ and $E_1$, we get $L=L_u$, and from
this $\csii$ is seen to satisfy condition \ref{def:regulus.2}.

\par
Finally, let $\cR$ be a partial $Z$-regulus which contains $\csii$. The partial
$Z$-reguli $\csii$ and $\cR$ have the same directrices, namely all lines that
meet three arbitrarily chosen elements of $\csii$ or, said differently, all
lines from $\csi$. We deduce from Lemma~\ref{lem:Zreg}, applied to an
arbitrarily chosen $X\in \cR$, that $X$ is a transversal subspace of $\csi$.
Now Remark~\ref{rem:segre}.\ref{rem:segre.1} gives $X\in\csii$, whence
$\cR=\csii$. This verifies that $\csii$ fulfills condition \ref{def:regulus.3}.
\end{proof}

\begin{thm}\label{thm:Zchain}
The $\Phi$-images of the $Z$-chains in $\Sigma(Z,R)$, $R=\End_K(U)$, are
exactly the $Z$-reguli in $\cG$.
\end{thm}

\begin{proof}
(a) By definition, all $Z$-chains comprise an orbit under the action of
$\GL(2,R)$. Because of Remark~\ref{rem:isoP(R)}.\ref{rem:isoP(R).3}, the
$\Phi$-images of $Z$-chains comprise an orbit under the action of the group
$\Aut_K(V)$. Clearly, all $f\in\Aut_K(V)$ map $Z$-reguli to $Z$-reguli. Hence
Proposition~\ref{prop:Sii} implies that the $\Phi$-image of any $Z$-chain of
$\bP(R)$ is a $Z$-regulus in $\cG$.

\par
(b) The group $\GL(2,R)$ acts transitively on the set of mutually distant
triplets of points of $\bP(R)$ (see \cite[Satz~1.3.8]{blu+he-05}). Due to
Remark~\ref{rem:isoP(R)}.\ref{rem:isoP(R).3}, we have a similar action of
$\Aut_K(V)$ on $\cG$. So, if we are given any $Z$-regulus $\cR$ then there
exists an $f\in\Aut_K(V)$ which takes three distinct (arbitrarily chosen)
elements of $\cR$ to $U=T^{(1,\,0)}$, $U'=T^{(0,\,1)}$, and $T^{(1,\,1)}$ (cf.\
also \cite[Lemma~2.1]{lue-80}). Clearly, $\cR^f$ is a $Z$-regulus. According to
Proposition~\ref{prop:Sii}, the set $\csii$ is a $Z$-regulus, too. The reguli
$\cR^f$ and $\csii$ have $T^{(1,\,0)}$, $T^{(0,\,1)}$, and $T^{(1,\,1)}$ in
common. This implies that $\cR^f$ and $\csii$ have the same set of directrices,
namely $\csi$. By Lemma~\ref{lem:Zreg}, any $X\in\cR^f$ is a transversal
subspace of $\csi$, so that Remark~\ref{rem:segre}.\ref{rem:segre.1} implies
$X\in\csii$. Now $\cR^f\subseteq\csii$ together with the maximality of $\cR^f$
yields $\cR^f=\csii$. By Proposition~\ref{prop:Sii}, the regulus $\csii$ is the
image of a $Z$-chain and, by virtue of $f^{-1}$, the same property holds for
$\cR$ according to (a).
\end{proof}

\begin{cor}\label{cor:orbit}
All $Z$-reguli of\/ $\cG$ comprise an orbit under the action of $\Aut_K(V)$.
Given any three mutually distant elements of $\cG$ there is a unique
$Z$-regulus containing them.
\end{cor}
By Proposition~\ref{prop:Sii} and Corollary~\ref{cor:orbit} the projectively
invariant properties of any $Z$-regulus $\cR$ can be read off from the
$Z$-regulus $\csii$. Below we state one such property. It is immediate from
(\ref{eq:1st}) and (\ref{eq:T(x,y)}) for the directrices of $\csii$, since
these are precisely the lines from $\csi$.

\begin{cor}\label{cor:subline}
Let $L$ be any directrix of a $Z$-regulus $\cR$. All points of $L$ which are
contained in an element of\/ $\cR$ form a $Z$-subline or, in other words, a
$Z$-chain of $L$.
\end{cor}
A notion of \emph{regulus} is introduced for any projective space over a (not
necessarily commutative) field $K$ in \cite{knarr-95}. Furthermore it is
pointed out that according to \cite{herz-72a} the existence of such a regulus
implies $K$ being equal to its centre $Z$ (see also \cite{grund-81}). As a
matter of fact, our conditions (R1) and (R2) mean the same as the identically
named conditions in \cite[p.~55]{knarr-95} together with a richness condition
stated there. Also, our directrices are precisely the \emph{transversals} in
the sense of \cite{knarr-95}. Corollary~\ref{cor:subline} shows that our
directrices satisfy the remaining condition (R3) in \cite{knarr-95} if, and
only if, $K=Z$. Hence for a commutative field $K$ the reguli in the sense of
\cite{knarr-95} are precisely our $Z$-reguli. Likewise, the reguli from
\cite{blunck-00a} coincide with our $Z$-reguli in this particular case, but
fail to have this property in case of a non-commutative ground field $K$. The
last assertion follows immediately from \cite[Lemma~4.1]{blunck-00a}.
\par
We proceed with two lemmas which will be needed in order to show that
$Z$-reguli can be defined in terms of the distant graph $(\cG,\dis)$.

\begin{lem}\label{lem:Z1}
Let $W, E_0, E\in\cG$ with $W\sim E_0$, $E\dis E_0$, $W\notdis E$. Then $W\cap
E$ is a point.
\end{lem}
\begin{proof}
Assume that $W\cap E$ contains a line $L$. Then $L\cap E_0=0$, because of
$E\dis E_0$. This implies that $\dim (W /(W\cap E_0))>1$, a contradiction to
$W\sim E_0$.

\par
Assume now that $W\cap E= 0 $. Since $W\sim E_0$, we have that $W=(W\cap E_0) +
p$ for some point $p\leq W$ with $p\not\leq E_0$. Then $p\not\leq E$, since
$W\cap E=0$. So there is a unique line $L$ through $p$ meeting $E_0$ and $E$.
Let $q_0=E_0\cap L$ and $q=E\cap L$. Then $q\not\leq W$, since $W\cap E=0$. So
also $q_0\not\leq W$, whence $E_0=(W\cap E_0)+q_0$. Moreover, $q_0\leq
L=p+q\leq W+E$, and we get $V=E_0+E=(W\cap E_0)+q_0+E\leq W+E$, a contradiction
to $W\notdis E$.

\par
Consequently, $W\cap E$ has to be a point.
\end{proof}

\begin{lem}\label{lem:Z2}
Let $E_0,E_1,E_2\in\cG$ be mutually distant, and let $W\in\cG$ satisfy $W\sim
E_0$, $W\notdis E_1, E_2$. Let $p_i=E_i\cap W$, $i\in\{1,2\}$, be the unique
intersection points according to Lemma \emph{\ref{lem:Z1}}. Then the line
$L=p_1+p_2$ meets $E_0$, i.e., $L$ is the unique line through $p_1$ meeting
$E_0$ and $E_2$.
\end{lem}
\begin{proof}
By definition, the line $L$ belongs to $W$. Since $W\sim E_0$, we have that
$W\cap E_0$ is a hyperplane in $W$. So $L$ must meet $W\cap E_0$.
\end{proof}

\begin{thm}\label{thm:Zreg}
A subset $\cR$ of\/ $\cG$ is a $Z$-regulus if, and only if, the following
conditions are satisfied:

\begin{cond}{$\dis$}
\item\label{thm:Zreg.1} $\cR$ is a distant clique with at least three
    elements.

\item\label{thm:Zreg.2} If three mutually distinct elements $E_0,E_1,E_2\in\cR$
    and any $W\in\cG$ satisfy $W\sim E_0$ and $W\notdis E_1,E_2$ then $W\notdis
    E$ for all $E\in\cR$.

\item\label{thm:Zreg.3} $\cR$ is not properly contained in any subset of $\cG$
    satisfying conditions \emph{\ref{thm:Zreg.1}} and \emph{\ref{thm:Zreg.2}}.

\end{cond}
\end{thm}
\begin{proof}
It suffices to show that the partial $Z$-reguli are precisely those subsets
$\cR$ of $\cG$ which satisfy \ref{thm:Zreg.1} and \ref{thm:Zreg.2}.

\par
First, let $\cR$ be a partial $Z$-regulus. Consider $E_0,E_1,E_2, W$ as in
\ref{thm:Zreg.2}. By Lemma \ref{lem:Z1}, the subspace $W\cap E_1$ is a point,
say $p_1$, and by Lemma \ref{lem:Z2}, $W$ contains the unique line $L$ through
$p_1$ meeting $E_0$ and $E_2$. Due to \ref{def:regulus.2}, each $E\in\cR$ meets
$L\leq W$, whence $E\notdis W$. So $\cR$ satisfies \ref{thm:Zreg.2} and clearly
also \ref{thm:Zreg.1}, since this condition coincides literally with
\ref{def:regulus.1}.

\par
Conversely, let $\cR$ be a subset of $\cG$ satisfying \ref{thm:Zreg.1} and
\ref{thm:Zreg.2}. Given any line $L$ that meets three distinct elements of
$\cR$, say $E_0, E_1,E_2$, we let $p_i=E_i\cap L$ for $i\in\{0,1,2\}$. Consider
the set
\begin{equation*}
    \cH:=\{H \mid p_0\leq H\leq E_0,\;\;\dim (E_0/H)=1\}.
\end{equation*}
This is the set of all hyperplanes of $E_0$ containing $p_0$. For each
$H\in\cH$ let $W_H=H+L$. Each $W_H$ belongs to the star $\cG[H\rangle$.
Consequently, $W_H$ satisfies $W_H\sim E_0$. Furthermore, for $i\in\{1,2\}$ we
have $W_H\notdis E_i$, since $p_i=L\cap E_i\leq W_H\cap E_i$.

\par
Let now $E$ be an arbitrary element of $\cR$ different from $E_0, E_1$. As
$\cR$ satisfies \ref{thm:Zreg.2}, we have that $W_H\notdis E$ holds for all
$W_H$ with $H\in \cH$. So by Lemma \ref{lem:Z2}, applied to $E_0, E_1, E$, we
obtain that all $W_H$ also contain the unique line $L'$ through $p_1$ meeting
$E_0$ and $E$. Since $p_0=\bigcap_{H\in\cH} H$, we have that $L'=L$. This
implies that $\cR$ satisfies \ref{def:regulus.2}, and clearly
\ref{def:regulus.1} is satisfied, too.
\end{proof}

\begin{thm}\label{thm:defbar}
The $\Phi$-images of $Z$-chains or, said differently, the $Z$-reguli can be
defined in terms of the distant graph $(\cG,\dis)$.
\end{thm}
\begin{proof}
This is immediate from Theorems~\ref{thm:Zchain} and \ref{thm:Zreg}, since the
formulation of \ref{thm:Zreg.2} only uses the relations $\dis$ and $\sim$, the
latter of which can be described with $\dis$ alone according to
Remark~\ref{rem:disadj}.
\end{proof}

\section{Consequences}\label{se:conseq}

This final section is devoted to the automorphism groups of the various
structures on $\cG$. The following corollaries are based on the observation
that two notions on $\cG$ give rise to the same automorphisms on $\cG$ if, and
only if, each of these notions is definable in terms of the other. Our first
result is a consequence of Theorem~\ref{thm:penadj} and the remarks preceding
that theorem:

\begin{cor}\label{cor:grassm-auto}
The automorphisms of the Grassmann graph $(\cG,\sim)$ are precisely the
collineations of the Grassmann space $(\cG,\fP)$.
\end{cor}
This corollary, which is part of the ``dimension-free'' theory, allows us to
draw several conclusions: Under any automorphism of the Grassmann graph maximal
adjacency cliques are preserved in both directions and, from
Corollary~\ref{cor:grassm-auto}, so are pencils. As mentioned in
Section~\ref{se:grass}, any star $\cG[M\rangle$ and any top $\cG \langle N]$ is
a singular subspace of $(\cG,\fP)$ which is isomorphic to the projective space
on $V/M$ and $N^*$ (the dual of $N$), respectively. However, for a closer
analysis we have to distinguish two cases:

\par
In the finite-dimensional case the automorphisms of the Grassmann graph are
precisely those bijections of $\cG$ onto itself which stem from semilinear
isomorphisms of $V$ onto itself or onto its dual $V^*$ (provided that $K$
admits an antiautomorphism). The two possibilities can be distinguished by
exhibiting the images of stars and tops: In the first case stars and tops are
preserved, in the second case they are interchanged. This is part of the
celebrated theorem of \textsc{W.~L.~Chow} \cite{chow-49}. See also
\cite{huanglp-09a}, \cite{huang-98}, \cite{kosiorek+m+p-08}, \cite{pankov-04b},
\cite[3.2.1]{pankov-10a}, \cite[Thm.~3.52]{wan-96}, \cite{west-74}, and the
references therein for proofs of Chow's initial result and various
generalisations.

\par
For infinite dimension the situation is different though: If we are given any
star $\cG[M\rangle$ and any top $\cG \langle N]$ then, due to $\dim V=\infty$,
both $M$ and $N$ belong to $\cG$. Consequently, $N\cong M\cong V/M$ (as vector
spaces). However, $\dim (V/M) = \dim N<\dim N^*$. Hence the projective spaces
on $\cG[M\rangle$ and $\cG \langle N]$ are non-isomorphic. Given any
automorphism $\kappa:\cG\to\cG$ of the Grassmann graph the Fundamental Theorem
of Projective Geometry (see, among others, \cite[1.4]{pankov-10a}) implies that
the restriction of $\kappa$ to any star and any top arises from a semilinear
isomorphism of the underlying vector spaces. This in turn allows us to deduce
that \emph{stars have to go over to stars and tops must go over to tops}.
However, an analogue of Chow's theorem fails to hold, as follows from the
subsequent example:

\begin{exa}\label{exa:nodistiso}
Choose $f\in\Aut_K(V)$ such that some $A\in\cG$ is mapped to $A^f\sim A$.
Define $\kappa:\cG\to\cG$ by $X^\kappa =X^f$ for all $X$ in the connected
component of $A$ and $X^\kappa:=X$ otherwise. This $\kappa$ is an automorphism
of the Grassmann graph, but it does not stem from any semilinear automorphism,
say $g$, of $V$. For, if there were such a $g$ then, on the one hand, we would
have $A^g=A^f\neq A$. On the other hand, all subspaces $Y\le A$ with
$\dim(A/Y)=1$ belong to $\cG$, but not to the connected component of $A$ by
(\ref{eq:zsh}). Therefore we would have $Y^g=Y$ and, $A$ being the union of all
such $Y$s, this would imply $A^g=A$, an absurdity. As a matter of fact our
example shows even more: The given $\kappa$ cannot be induced by any bijection
$g:V\to V$ such that $g$ and $g^{-1}$ preserve subspaces belonging to $\cG$,
let alone $g$ being semilinear.
\end{exa}
The previous example is based on the fact that the Grassmann graph is
disconnected precisely when $\dim V=\infty$. (See Remark~\ref{rem:disconn}.)
Without going into details let us just mention that the related algebraic
result about $R=\End_K(U)$ is as follows: All linear endomorphisms of $U$ with
finite rank comprise a proper two-sided ideal of $R$ precisely when $\dim
U=\infty$. See, among others, \cite[p.~164]{ander+f-74},
\cite[pp.~197--199]{baer-52}, and \cite{orsatti+r-95}.

\par
In the infinite-dimensional case an explicit description of all automorphisms
of the Grassmann graph $(\cG,\sim)$ seems to be unknown. On the other hand, an
analogue of Chow's theorem holds for the automorphism group of the Grassmann
graph formed by all subspaces of a fixed \emph{finite} dimension of an
infinite-dimensional vector space \cite{lim-10}.

\par
We now turn to the distant graph $(\cG,\dis)$. Let us write $\fR$ for the set
of all $Z$-reguli in $\cG$. Then the (non-linear) incidence geometry
$(\cG,\fR)$ is a model of the chain geometry $\Sigma(Z,R)$, and we call it the
\emph{space of $Z$-reguli} on $\cG$. This generalises a notion from
\cite[Kap.~10]{blu+he-05}. By specialising Remark~\ref{rem:dist-chain} to
$F=Z$, we see that two distinct elements of $\bP(R)$ are distant if, and only
if, they are on a common $Z$-chain. Together with Theorem~\ref{thm:Zchain} and
Theorem~\ref{thm:defbar} we therefore obtain:

\begin{cor}\label{cor:reguli-auto}
The automorphisms of the distant graph $(\cG,\dis)$ are precisely the
automorphisms of the space $(\cG,\fR)$ of\/ $Z$-reguli.
\end{cor}

\begin{cor}\label{cor:chain-auto}
The automorphisms of the distant graph $(\bP(R),\dis)$ are precisely the
automorphisms of the chain geometry $\Sigma(Z,R)$.
\end{cor}
According to Remark~\ref{rem:disadj}.\ref{rem:disadj.1} any automorphism of the
distant graph is also an automorphism of the Grassmann graph on $\cG$. There
are two cases:

\par
In the finite-dimensional case a complete description of the automorphisms of
$(\cG,\dis)$ can be derived from Chow's theorem: See \cite[Thm.~4.4]{blu+h-03b}
and cf.\ \cite{huanglp-09a}, \cite{huang-10a}, \cite{huang+h-08a},
\cite{lim-10} for generalisations. Furthermore, the results from
\cite[Thm.~5.4]{blunck+h-05b} provide an explicit description of the
automorphisms of the chain geometry $\Sigma(Z,R)$, thereby avoiding the
richness condition appearing in the related result from
\cite[Kor.~4.3.10]{blu+he-05}.

\par
For infinite dimension of $V$ the only known automorphisms of the distant graph
on $\cG$ seem to be those which stem from semilinear automorphisms of $V$. Thus
in this case there remains the problem of finding all automorphisms of
$(\cG,\dis)$.

\subsubsection*{Acknowledgements} The first author was partially supported by a
grant from Vienna University of Technology (International Office). She thanks
the members of the Research Group Differential Geometry and Geometric
Structures at TU Vienna for their hospitality.


\noindent Andrea Blunck\\
Fachbereich Mathematik\\
Universit\"{a}t Hamburg\\
Bundesstra{\ss}e 55\\
D-20146 Hamburg\\
Germany\\
\texttt{andrea.blunck@math.uni-hamburg.de}
\\~\\
\noindent
Hans Havlicek\\
Institut f\"{u}r Diskrete Mathematik und Geometrie\\
Technische Universit\"{a}t\\
Wiedner Hauptstra{\ss}e 8--10/104\\
A-1040 Wien\\
Austria\\
\texttt{havlicek@geometrie.tuwien.ac.at}


\end{document}